\documentclass{article}
\usepackage[dvips]{graphicx}
\usepackage{amsmath}
\usepackage{amsfonts}
\usepackage{amsthm}
\usepackage{amssymb}
\usepackage{mathrsfs}
\usepackage[T1]{fontenc}
\usepackage[latin1]{inputenc}
\usepackage[all]{xy}
\usepackage{a4wide}
\usepackage{psfrag}
\usepackage[bf]{caption}
\usepackage{mathtools}
\usepackage{lipsum}

\newcommand{\tykkp}{\mathbb{P}}
\newcommand{\fino}{\mathcal{O}}

\newcommand{\fineca}{\mathcal{E}_{C,A}}
\newcommand{\finixi}{\mathcal{I}_{\xi}}
\newcommand{\Hom}{\mathrm{Hom}}
\newcommand{\finfca}{\mathcal{F}_{C,A}}

\newcommand{\fing}{\mathcal{G}}
\newcommand{\fini}{\mathcal{I}}
\newcommand{\length}{\mathrm{length}}

\newcommand{\Cliff}{\mathrm{Cliff}}
\newcommand{\rk}{\mathrm{rk}}
\newcommand{\Nvee}{N^{\vee}}
\newcommand{\Mvee}{M^{\vee}}

\makeatletter
\newcommand{\subjclass}[2][2010]{%
  \let\@oldtitle\@title%
  \gdef\@title{\@oldtitle\footnotetext{#1 \emph{Mathematics subject classification.} #2}}%
}
\newcommand{\keywords}[1]{%
  \let\@@oldtitle\@title%
  \gdef\@title{\@@oldtitle\footnotetext{\emph{Key words.} #1.}}%
}
\makeatother
\newtheorem{Prop}{Proposition}[section]
\newtheorem{Theorem}[Prop]{Theorem}
\newtheorem{Lemma}[Prop]{Lemma}

\newtheoremstyle{Conjecture}{\topsep}{\topsep}%
  {}
  {}
  {\bfseries}
  {.}
  { }
  {\thmname{#1}\thmnumber{ #2}\thmnote{ #3}}

\theoremstyle{definition}

\newtheorem{Conjecture}[Prop]{Conjecture}

\newtheorem{Remark}[Prop]{Remark}

\AtEndDocument{\bigskip{\footnotesize%
  \textsc{Dept.\ of Mathematics and Natural Sciences, University College of Southeast Norway}  \par
  \textit{E-mail address}: \texttt{nils.h.rasmussen@usn.no} \par
}}

\author{Nils Henry Rasmussen}


\begin{document}

\title{Pencils and nets on curves arising from rank 1 sheaves on K3 surfaces}

\maketitle


\begin{abstract}
Let $S$ be a K3 surface, $C$ a smooth curve on $S$ with $\fino_S(C)$ ample, and $A$ a base-point free $g^2_d$ on $C$ of small degree. We use Lazarsfeld--Mukai bundles to prove that $A$ is
cut out by the global sections of a rank 1 torsion-free sheaf $\fing$ on $S$. Furthermore, we show that $c_1(\fing)$ with one exception is adapted to $\fino_S(C)$ and satisfies
$\Cliff(c_1(\fing)_{|C})\leq\Cliff(A)$, thereby confirming a conjecture posed by Donagi and Morrison. We also show that the same methods can be used to give a simple proof of the conjecture
in the $g^1_d$ case.

In the final section, we give an example of the mentioned exception where $h^0(C,c_1(\fing)_{|C})$ is dependent on the curve $C$ in its linear system, thereby failing to be adapted to
$\fino_S(C)$.
\end{abstract}

\section{Introduction}

In the past 30 years, one central problem in the study of the existence of $g^r_d$'s on smooth curves has been to find connections between sheaves on K3 surfaces $S$ and linear systems on
curves $C$ lying on $S$. This started with Lazarsfeld (\cite{Lazarsfeld}) and Tyurin (\cite{Tyurin}) independently introducing vector-bundles $\fineca$ on $S$, depending on a smooth curve
$C$ and base-point free complete linear system $A$ on $C$, providing much information on the geometry of the curve $C$ and existence of other linear systems on the curve.

These vector-bundle techniques have given grounds for many results. Among these, Green and Lazarsfeld (\cite{G-L}) proved in 1987 that the Clifford index is constant for all smooth curves
$C$ in a linear system on a K3 surface, and that it is maximal (i.e., $=\lfloor(g-1)/2\rfloor$) if the Picard group is generated by $\fino_S(C)$. In 1995, Ciliberto and Pareschi
(\cite{Ciliberto}) proved that the gonality is constant for all smooth curves in an ample linear system on a K3 surface. Knutsen (\cite{Knutsen}) proved that this is also true for non-ample
linear systems, except for one particular case. Furthermore, he proved that there exist only two examples of exceptional curves. Among other notable results, Aprodu and Farkas
(\cite{Aprodu-Farkas}) proved in 2011 that the Green conjecture is satisfied for all smooth curves on K3 surfaces. They also found the exact dimension of $g^1_d$'s for the general curves in
a linear system.


Lelli-Chiesa (\cite{lelli2015generalized}) proved a conjecture posed by Donagi and Morrison (\cite{D-M}), in the case of K3 surfaces without $(-2)$ curves, $d\leq g-1$ and
$\Cliff(A)=\Cliff(C)$. The conjecture is stated as follows:

\begin{Conjecture}[(Donagi--Morrison, \cite{D-M})]
Suppose $C$ is a smooth curve on a K3 surface $S$, and let $A$ be a base-point free complete $g^r_d$ on $C$ such that $\rho(g,r,d)<0$. Then there exists a line bundle $D$ on $S$, adapted to
$\fino_S(C)$, such that $A\leq D_{|C}$ and $\Cliff(D_{|C})\leq \Cliff(A)$.
\end{Conjecture}

Here, the \emph{Clifford index} of a line bundle $A$ on a smooth curve $C$ is defined as $\Cliff(A):=\deg(A)-2(h^0(C,A)-1)$. We also mention that
$\Cliff(C):=\min\{\Cliff(A)\,|\,h^0(C,A),h^1(C,A)\geq 2\}$ (but where $\Cliff(C)$ is defined to be $0$ for hyperelliptic curves of genus $2$ or $3$, and $1$ for trigonal curves of genus
$3$). The
value $\rho(g,r,d)$ is the \emph{Brill--Noether number} and is defined as $\rho(g,r,d):=g-(r+1)(g-d+r)$. A line bundle $D$ on $S$ is said to be \emph{adapted to the line bundle} $L$ if: 
\begin{itemize}
\item[(i)] $h^0(S,D)\geq 2$ and $h^0(S,L\otimes D^{\vee})\geq 2$, and
\item[(ii)] $h^0(C,D_{|C})$ is independent of the curve $C\in |L|_s$;
\end{itemize}
where $|L|_s$ denotes smooth curves in $|L|$.

The conjecture was proved in \cite{D-M} for the case of $g^1_d$'s, and basically involved proving that $c_1$ of the cokernel of the maximal destabilising sequence of $\fineca$ satisfies the
conditions of the line bundle $D$ in the conjecture, with the exception of one special case.

Part of the conjecture was also proved by Lelli-Chiesa in \cite{lelli2013stability} for the case of $g^2_d$'s on curves on maximal gonality and Clifford dimension 1. There, the idea was to
prove that the kernel of the maximal destabilising sequence of $\fineca$ can be assumed to be of rank $1$, and that the determinant of the cokernel is the desired line bundle $D$ of the
conjecture.

In the proof of our result, we use similar ideas. The main result states that the divisors in base-point free complete $g^2_d$'s for small $d$ are equal to global sections of torsion-free
sheaves of rank $1$ on $S$ restricted to $C$. The torsion-free sheaves arise naturally from a maximal destabilising sequence of $\finfca=\fineca^{\vee}$, and $c_1$ of these sheaves satisfy
the conditions on $D$ of the conjecture, similar to what is done in Donagi--Morrison's and Lelli-Chiesa's proofs.

Our main result is the following:

\begin{Theorem}\label{main}
Let $S$ be any K3 surface, and let $L$ be an ample line-bundle on $S$. If $C\in |L|$ is smooth and $A$ is a base-point free complete $g^2_d$ on $C$ satisfying $d\leq\frac{1}{6}L^2$, then
the following is satisfied:
\begin{itemize}
\item[(a)] There exists a linear system $|D|$ on $S$ and a finite subscheme $\xi\subset S$ such that every divisor in $|A|$ is equal to an element in $|D\otimes\finixi|$ restricted to
    $C$, where $\finixi$ is the ideal sheaf of $\xi$.
\end{itemize}
Suppose furthermore that there do not exist an elliptic pencil $E$ and a $(-2)$ curve $\Gamma$ satisfying both conditions $A=E^{\otimes 2}_{|C}$ and $(L\otimes E^{\otimes (-2)}).\Gamma=-2$.
Then:
\begin{itemize}
\item[(b)] The line bundle $D$ found in (a) is adapted to $L$; and
\item[(c)] $\Cliff(D_{|C})\leq\Cliff(A)$.
\end{itemize}
\end{Theorem}

\begin{Remark}
In the case where there exist an elliptic pencil $E$ and a $(-2)$ curve $\Gamma$ satisfying $A=E^{\otimes 2}_{|C}$ and $(L\otimes E^{\otimes (-2)}).\Gamma=-2$, we can construct examples
where $h^0(C,E^{\otimes 2}_{|C})$ is dependent on the curve $C$ in $|L|_s$. See Section \ref{example}.
\end{Remark}

The tools used in the proof of Theorem \ref{main}, can also be used to give a simple proof of Donagi and Morrison's result for the $g^1_d$ case (\cite[Theorem 5.1']{D-M}). Here, we avoid
the special case that was considered in the original proof. Furthermore, as in Theorem \ref{main}, we also here prove that all divisors in $|A|$ are equal to the restriction to $C$ of the
global sections of a rank-$1$ torsion-free sheaf on $S$.

\begin{Theorem}\label{main2}
Let $|L|$ be any base-point free linear system on a K3 surface $S$. If $C\in |L|$ is smooth and $A$ is a base-point free complete $g^1_d$ on $C$ satisfying $\rho(g,1,d)<0$, then the
following is
satisfied: 
\begin{itemize}
\item[(a)] There exists a linear system $|D|$ on $S$ and a finite subscheme $\xi\subset S$ such that every divisor in $|A|$ is equal to an element in $|D\otimes\finixi|$ restricted to
    $C$, where $\finixi$ is the ideal sheaf of $\xi$;
\item[(b)] the line bundle $D$ found in (a) is adapted to $L$; and
\item[(c)] $\Cliff(D_{|C})\leq\Cliff(A)$.
\end{itemize}
\end{Theorem}

We will be working in characteristic $0$ throughout this paper.

\section{Proof of theorem}

The Lazarsfeld--Mukai vector bundles are defined as follows: Given a smooth curve $C$ of genus $g$ on $S$ and a base-point free, complete $g^r_d$ $A$ on $C$, the vector-bundle $\finfca$ on
$S$ is defined as the kernel of the evaluation morphism $H^0(C,A)\otimes\fino_S\rightarrow A \rightarrow 0$. The bundle has the following properties:
\begin{itemize}
\item $\mathrm{rk}(\finfca)=r+1$.
\item $\mathrm{det}(\finfca)=\fino_S(-C)$.
\item $c_2(\finfca)=d$.
\item $h^0(S,\finfca)=h^1(S,\finfca)=0$.
\item $\chi(S,\finfca)=2-2\rho(g,r,d)$ where $\rho(g,r,d)=g-(r+1)(g-d+r)$ is the Brill--Noether number.
\item The dual, $\finfca^{\vee}$, is globally generated away from a finite set.
\end{itemize}

Note that if $\rho(g,r,d)<0$, then $2h^0(S,\finfca\otimes\finfca^{\vee})\geq\chi(S,\finfca)\geq 4$, and so $\finfca$ is then non-simple, and hence non-stable.

We will for the remainder of this paper -- except in the proof of Theorem \ref{main2} -- assume that $A$ is a base-point free, complete $g^2_d$ on a smooth curve $C\in |L|$ with $L$ ample
and $d\leq\frac{1}{6}L^2$. By \cite[Theorem 3.4.1]{huybrechts2010geometry}, $\finfca$ is then unstable, and there thus exists a maximal destabilising sequence
\begin{equation}\label{maxdestab}
0\rightarrow M\rightarrow\finfca\rightarrow N\rightarrow 0
\end{equation}
such that $M$ is locally free, $N$ is torsion-free and $\mu_L$-semistable, and $c_1(M).L>-\rk(M)\frac{1}{3}L^2$. 

In the statements that follow, we will also be needing the dualisation of this sequence, which is
\begin{equation}\label{maxdestabdual}
0\rightarrow\Nvee\rightarrow\finfca^{\vee}\rightarrow\tilde{M}\rightarrow 0,
\end{equation}
where $\tilde{M}$ is torsion-free and satisfies $\tilde{M}^{\vee}=M$. Since $\finfca^{\vee}$ is globally generated away from a finite set, the same applies for $\tilde{M}$. This sequence is
maximal destabilising for $\finfca^{\vee}$, and so $\tilde{M}$ must be $\mu_L$-semistable.

The following lemma is needed for the proof of Proposition \ref{rk2}, which (among other things) states that we can assume the rank of $M$ to be $2$. This is the most important step for the
proof of Theorem \ref{main}.

\begin{Lemma}\label{mngeq0}
Let $A$, $C$ and $L$ be as above, and consider the maximal destabilising sequence \eqref{maxdestab} of $\finfca$. If $\rk(M)=1$, then $M.c_1(N)\geq 0$.
\end{Lemma}

\begin{proof}
Suppose $M.c_1(N)=\Mvee.c_1(\Nvee)<0$.

We dualise the sequence \eqref{maxdestab}, yielding
$$0\rightarrow N^{\vee}\rightarrow\finfca^{\vee}\rightarrow M^{\vee}\otimes\fini_{\eta}\rightarrow 0,$$
where $\fini_{\eta}$ is the ideal sheaf of a $0$-dimensional subscheme $\eta$. Since $\finfca^{\vee}$ is globally generated away from a finite set, then so is $M^{\vee}$, and it follows
that a sufficient condition for $M.c_1(N)=M^{\vee}.c_1(N^{\vee})$ to be $\geq 0$ is that $h^0(S,c_1(N^{\vee}))\geq 1$.

Now, since $M.L>-\frac{1}{3}L^2$ and $M\otimes c_1(N)\cong L^{\vee}$, then $c_1(N).L<-\frac{2}{3}L^2<0$, and so it suffices to show that $c_1(N)^2\geq 0$, since it then follows that either
$c_1(N)$ or
$c_1(\Nvee)$ must be effective, and we see that it must be $c_1(\Nvee)$. 

Now, since we have found that $c_1(N).L<-\frac{2}{3}L^2$, then $c_1(\Nvee).L>\frac{2}{3}L^2$. This gives us the inequality $c_1(\Nvee)^2+c_1(\Nvee).\Mvee=c_1(\Nvee).L>\frac{2}{3}L^2$. And
since we are
assuming that $c_1(\Nvee).\Mvee <0$, it follows that $c_1(\Nvee)^2 >0$. 

The result follows.
\end{proof}

\begin{Prop}\label{rk2}
Let $A$, $C$ and $L$ be as above, and let \eqref{maxdestab} be a maximal destabilising sequence of $\finfca$. Then $\rk(M)=2$; $c_1(M)^2\geq 0$; $c_2(\tilde{M})\geq 0$, where $\tilde{M}$ is
as in \eqref{maxdestabdual}; and $c_2(M)\geq 0$.
\end{Prop}

\begin{proof}
Suppose $\rk(M)=1$. Then $N$ is semistable of rank $2$, and so by \cite[Theorem 3.4.1]{huybrechts2010geometry}, $c_2(N)\geq \frac{1}{4}c_1(N)^2$. Furthermore, $M.L>-\frac{1}{3}L^2$, and so
$c_1(N).L<-\frac{2}{3}L^2$. From \eqref{maxdestab}, we hence get $c_2(\finfca)\geq
M.c_1(N)+\frac{1}{4}c_1(N)^2=\frac{1}{4}c_1(N)(c_1(N)+4M)=\frac{1}{4}c_1(N)(-L+3M)>\frac{2}{12}L^2+\frac{3}{4}c_1(N).M$, and by Lemma \ref{mngeq0}, this is $\geq\frac{1}{6}L^2$. Since
$c_2(\finfca)=\deg(A)$, which was assumed to be $\leq\frac{1}{6}L^2$, this gives the desired contradiction.

To prove the first two inequalities of the statement, we consider \eqref{maxdestabdual} and note that since $\tilde{M}$ is globally generated away from a finite set, then the same must
apply for $c_1(\tilde{M})=c_1(\Mvee)$, and so $c_1(M)^2=c_1(\Mvee)^2\geq 0$. By \cite[Theorem 3.4.1]{huybrechts2010geometry}, we must have $c_2(\tilde{M})\geq 0$ as a consequence.

We now prove the last statement. Since $\Mvee$ is globally generated away from a finite set, we can inject an effective line-bundle $D_1$ into $\Mvee$, assume that the injection is
saturated, and get
$$0\rightarrow D_1\rightarrow \Mvee\rightarrow D_2\otimes \fini_{\eta}\rightarrow 0,$$
where $\eta$ is a possibly empty zero-dimensional subscheme and $D_2$ a line-bundle. Since $\Mvee$ is globally generated away from a finite set, then $D_2$ is also globally generated
(actually everywhere since it is base-component free). But then, $D_1.D_2\geq 0$, and dualising, we see that $c_2(M)\geq 0$.
\end{proof} 

In the proof of part (c) of Theorem \ref{main}, we will be needing the following result:

\begin{Prop}\label{HIT}
Suppose $D_1$ and $D_2$ are two divisors on a K3 surface, and suppose $D_2^2>0$. Then $D_1^2D_2^2\leq (D_1.D_2)^2$, with equality if and only if $(D_1.D_2)D_1\sim D_1^2D_2$.
\end{Prop}

\begin{proof}
This follows from the Hodge Index Theorem (see e.g.\ \cite[Corollary 2.16]{Barth}) and \cite[Chapter 1, Exercise 10]{Friedman}, and using that numeric and linear equivalence are the same
for divisors on a K3 surface.
\end{proof}

We now give the proof of Theorem \ref{main}.

\begin{proof}[Proof of Theorem \ref{main}] We begin by proving part (a) of the theorem.

Let $A$, $C$ and $L$ be as in the theorem, and let $\finfca$ be the associated Lazarsfeld--Mukai bundle. Since $\deg(A)\leq\frac{1}{6}L^2$, it follows from \cite[Theorem
3.4.1]{huybrechts2010geometry} that $\finfca$ is unstable, and so we obtain a maximal destabilising sequence \eqref{maxdestab}.

The injection $M \hookrightarrow \finfca$ can be composed with $\finfca\hookrightarrow \fino_S^{\oplus 3}$, yielding the following diagram, where $\fing$ is the cokernel:

\begin{equation}\label{diagram}
\xymatrix{0 \ar[r] & M \ar[r]\ar[d] & \fino_S^{\oplus 3} \ar[r]^<<<<<{\widetilde{\mathrm{ev}}}\ar@{=}[d] & \fing \ar[r]\ar[d]^{\phi} & 0 \\
            0 \ar[r] & \finfca \ar[r] & \fino_S^{\oplus 3} \ar[r]^<<<<<{\mathrm{ev}} & A \ar[r] & 0.
}
\end{equation}

By the snake lemma, $\ker(\phi)\cong N$, and since any torsion element of $\fing$ must map to $0$ in $A$ and $N$ is torsion-free, it follows that $\fing$ is torsion-free. Since $\rk(M)=2$,
by Proposition \ref{rk2}, we have $\rk(\fing)=1$, and so it follows that $\fing=D\otimes\finixi$, where $D=c_1(\Mvee)$ and $\finixi$ is the ideal sheaf of a possibly empty finite subscheme
$\xi$. Furthermore, $\phi$ is injective on global sections, since $h^0(S,N)=0$ (proof: If $h^0(S,N)>0$, then -- since $\rk(N)=1$ -- we can write $N\cong N'\otimes \fini_{\zeta}$ where
$\fini_{\zeta}$ is the ideal sheaf of a
possibly empty finite subscheme $\zeta$, and where $N'$ is an effective line bundle. But then $c_1(N)=N'.L\geq 0$, contradicting \eqref{maxdestab} being a maximal destabilising sequence). 

It is clear that $h^0(S,M)=0$ since $h^0(S,\finfca)=0$, and so $h^0(S,\fing)\geq 3=h^0(C,A)$. Since $\phi$ is injective on global sections, this means that $h^0(S,\fing)=3$, and so each
global section of $A$ comes from a unique global section of $\fing$. The map $\phi$ is an element of $\Hom(D\otimes\finixi,A)=\Hom(\finixi,A\otimes D^{\vee})$, implying that either $A\cong
D_{|C}\otimes\fini_{\xi'}$, with $\xi'=\xi\cap C$, or $h^0(C,A)>h^0(C,D\otimes\fini_{\xi'})$. However, we have $h^0(C,A)=h^0(S,D\otimes\finixi)$, and having
$h^0(S,D\otimes\finixi)>h^0(C,D\otimes\fini_{\xi'})$ would imply that
$h^0(S,D\otimes\fino_S(-C))>0$, which contradicts that $c_1(M).L>-\frac{2}{3}L^2$. 

We conclude that $A\cong D_{|C}\otimes\fini_{\xi'}$, and that every divisor in $|A|$ comes from restricting divisors in $|D|$ vanishing on $\xi$.\\ 
\\
Proof of part (b):

We recall that $D$ is by definition adapted to $L$ if $h^0(S,D)\geq 2$, $h^0(S,L\otimes D^{\vee})\geq 2$, and $h^0(C,D_{|C})$ is independent of the curve $C$ in $|L|_s$.

We already know that $h^0(S,D\otimes\finixi)=3$, so the first condition is clear. To show that $h^0(S,L\otimes D^{\vee})\geq 2$, note that $D\otimes L^{\vee}\cong c_1(N)$, which cannot be
effective since $c_1(N).L<-\frac{1}{3}L^2$. It thus suffices, by Riemann--Roch, to show that $c_1(\Nvee)^2\geq 0$. We have
$c_1(\Nvee)^2=(L-c_1(\Mvee)).c_1(\Nvee)=L.c_1(\Nvee)-c_1(\Mvee).c_1(\Nvee)>\frac{1}{3}L^2-c_1(\Nvee).c_1(\Mvee)$.  Sequence \eqref{maxdestabdual} gives us $\frac{1}{6}L^2\geq
c_2(\finfca^{\vee})=c_1(\Mvee).c_1(\Nvee)+c_2(\tilde{M})\geq c_1(\Mvee).c_1(\Nvee)$ (using Proposition \ref{rk2} on $c_2(\tilde{M})$, and that
$c_2(\Nvee)=0$ since it is locally free of rank $1$), and so we can conclude that $c_1(\Nvee)^2$ is nonnegative (actually, it is $\geq \frac{1}{6}L^2$). 

We now show that $h^0(C,D_{|C})$ is independent of the curve $C$ in $|L|_s$. By taking cohomology of the sequence $0\rightarrow D\otimes L^{\vee}\rightarrow D\rightarrow D_{|C}\rightarrow
0$, we see that it suffices to show that $h^1(S,D)=0$. We saw in Proposition \ref{rk2} that $c_1(M)^2$, which is the same as $D^2$, is nonnegative, and so $h^1(S,D)>0$ if and only if there
exists a $-2$-curve $\Gamma$ such that $\Gamma.D<0$ or $D=\fino_S(nE)$ for some positive integer $n$ and where $E$ is an elliptic curve (see \cite[Theorem]{knutsen2007sharp}). Since
$\tilde{M}$ is globally generated away from a finite set, then so is $D$ (it is actually base-point free, since it is on a K3 surface and has no base-components), and so no $-2$-curve can
intersect $D$ negatively. In order to prove that
$h^1(S,D)=0$, it therefore suffices to prove that $D^2>0$. 

To prove that $D^2>0$, the top row of \eqref{diagram} shows that $0=c_1(M).c_1(\fing)+c_2(M)+c_2(\fing)=-c_1(M)^2+c_2(M)+c_2(\fing)$. Since $c_2(M)\geq 0$ by Proposition \ref{rk2}, then
$c_1(M)^2\geq c_2(\fing)$. We obviously have $c_2(\fing)\geq 0$, and we have $c_1(M)^2=0$ only if $c_2(\fing)=0$, and hence only if $\fing=D$, with $h^0(S,D)=3$. However, in this case,
$A=D_{|C}$, and we clearly see that $D=E^{\otimes 2}$ where $E$ is an elliptic pencil.

We prove that $h^0(C',E^{\otimes 2}_{|C'})=3$ for all $C'\in |L|_s$ under the conditions of the theorem, given that there exists a curve $C\in |L|_s$ where $h^0(C,E^{\otimes 2}_{|C})=3$.
Consider the exact sequence
$$0\rightarrow E^{\otimes 2}\otimes L^{\vee}\rightarrow E^{\otimes 2}\rightarrow E^{\otimes 2}_{|C}\rightarrow 0.$$ 
If we take cohomology, we see that since $h^1(S,E^{\otimes 2})=1$, $h^0(S,E^{\otimes 2})=3$, $h^0(C,E^{\otimes 2}_{|C})=3$, and $h^0(S,E^{\otimes 2}\otimes L^{\vee})=0$ (the latter as a
consequence of $L.(E^{\otimes 2}\otimes L^{\vee})=L.(c_1(\Mvee)\otimes L^{\vee})<-\frac{1}{3}L^2$), we must have $h^1(S,E^{\otimes 2}\otimes L^{\vee})\leq 1$. If $h^1(S,E^{\otimes 2}\otimes
L^{\vee})=0$, then $h^0(C',E^{\otimes 2}_{|C'})=3$ for all $C'\in |L|_s$, and we are done. If $h^1(S,E^{\otimes 2}\otimes L^{\vee})=1$, then so is $h^1(S,L\otimes E^{\otimes (-2)})$, and by
\cite[Theorem]{knutsen2007sharp}, we either have $(L\otimes E^{\otimes (-2)})^2=0$ or there exists a $(-2)$-curve $\Gamma$ such that $(L\otimes E^{\otimes (-2)}).\Gamma\leq -2$. From the
condition $E^{\otimes 2}.L\leq \frac{1}{6}L^2$, we have $(L\otimes E^{\otimes (-2)})^2\geq\frac{2}{3}L^2>0$, and so we are in the case where there exists a $(-2)$-curve $\Gamma$ such that
$(L\otimes E^{\otimes (-2)}).\Gamma\leq -2$. 

We prove that $(L\otimes E^{\otimes (-2)}).\Gamma=-2$, thus contradicting the conditions of the theorem. If $(L\otimes E^{\otimes (-2)}).\Gamma<-2$, then $\Gamma^{\otimes 2}$ must be a base
component of $L\otimes E^{\otimes (-2)}$. However, since $h^0(S,L\otimes E^{\otimes (-2)})=h^0(S,L\otimes E^{\otimes (-2)}\otimes \Gamma^{\otimes (-2)})$, then
$\chi(S,L\otimes E^{\otimes (-2)}\otimes \Gamma^{\otimes (-2)})-\chi(S,L\otimes E^{\otimes (-2)})\leq h^1(S,L\otimes E^{\otimes (-2)})=1$ (noting that no $h^2$ terms are positive because of
the condition $E^{\otimes 2}.L\leq\frac{1}{6}L^2$); and Riemann--Roch gives us $-2(L\otimes E^{\otimes (-2)}).\Gamma-4\leq 1$, which is impossible.\\
\\
Proof of part (c):

We have $\Cliff(A)=d-4=c_2(\finfca)-4$. We must prove that $\Cliff(D_{|C})$ is at most equal to this.

By definition, $\Cliff(D_{|C})=D.L-2(h^0(C,D_{|C})-1)$. Since $D=c_1(\Mvee)$, $D.L=c_1(\Mvee).c_1(\Nvee)+c_1(\Mvee)^2$ and $h^0(C,D_{|C})\geq h^0(S,D)\geq \frac{1}{2}c_1(\Mvee)^2+2$, this
immediately gives us
\begin{equation}\label{Cliff(G)}
\Cliff(D_{|C})\leq c_1(\Mvee).\Nvee-2=c_2(\finfca)-c_2(\tilde{M})-2,
\end{equation}
where we recall that $\tilde{M}$ is as given in \eqref{maxdestabdual}. We must prove that $c_2(\tilde{M})\geq 2$.

Since $\tilde{M}$ is semistable, it follows from \cite[Theorem 3.4.1]{huybrechts2010geometry} that $c_2(\tilde{M})\geq \frac{1}{4}c_1(\tilde{M})^2$. In part (b) of the proof, we showed that
$c_1(\tilde{M})^2>0$ (since $D=c_1(\tilde{M})$) except for the case when $D_{|C}=A$. Since the result follows trivially in this latter case, we can assume that $c_1(\tilde{M})^2>0$, and so
$c_2(\tilde{M})\geq 1$. In the following, we will suppose $c_2(\tilde{M})=1$ (and hence $c_1(\tilde{M})^2\leq 4$) and show that this yields a contradiction.


First note that, by taking cohomology of \eqref{maxdestabdual} and recalling that $h^1(S,\finfca^{\vee})=h^2(S,\finfca^{\vee})=0$, we see that $h^1(S,\tilde{M})=h^2(S,\tilde{M})=0$ (we
proved that $h^2(S,\Nvee)=h^0(S,N)=0$ in part (a)). Also, since $\tilde{M}$ is of rank $2$ and globally generated away from a finite set, it must sit inside an exact sequence
\begin{equation}\label{M-sekvens}
0\rightarrow R_1\otimes\fini_{\nu}\rightarrow\tilde{M}\rightarrow R_2\otimes\fini_{\eta}\rightarrow 0,
\end{equation}
where $R_i$ are line-bundles and $\nu$ and $\eta$ finite subschemes. We can furthermore assume that $R_1$ is effective since $\tilde{M}$ has global sections, and $R_2$ is globally generated
and $R_2\otimes\fini_{\eta}$ globally generated away from a finite set; and hence $R_1.R_2\geq 0$, $R_2^2\geq 0$, and $\mathrm{length}(\nu),\mathrm{length}(\eta)\leq 1$. Note that
$R_1.R_2=1-\mathrm{length}(\eta)-\mathrm{length}(\nu)$. Also, \cite[Theorem]{knutsen2007sharp} gives us that $h^1(S,c_1(\tilde{M}))=0$. 

\paragraph{Case: $\mathrm{length}(\eta)=1$.} In this case, $R_1.R_2=\mathrm{length}(\nu)=0$, and since at least one $R_i$ must satisfy $R_i^2>0$, Proposition \ref{HIT} yields that $R_1^2R_2^2\leq 0$,
and so either $R_1^2<0$, or $R_1^2=0$. (If $R_1^2>0$ with $R_2^2=0$, we get $R_2=\fino_S$, and then $R_2\otimes\fini_{\eta}$ has no global sections.) If $R_1^2<0$, then also
$R_1.c_1(\tilde{M})=R_1.(R_1\otimes R_2)<0$, and so $R_1$ is a base component of $c_1(\tilde{M})$. However, since $\tilde{M}$ is globally generated away from a finite set, then so
must $c_1(\tilde{M})$, and we get a contradiction. 

If $R_1^2=0$, then $(R_1\otimes R_2)^2=R_2^2$, and putting $\fino_S(D_1)=R_2$ and $\fino_S(D_2)=R_1\otimes R_2$, we get equality in Proposition \ref{HIT}, and so $R_1=\fino_S$ (since
numerical and linear equivalence is the same for line bundles on K3 surfaces). However, in that case, taking cohomology of \eqref{M-sekvens} gives us that $h^1(S,R_2\otimes\fini_{\eta})=1$,
while cohomology of the sequence
$$0\rightarrow R_2\otimes\fini_{\eta}\rightarrow R_2\rightarrow \fino_{\eta}\rightarrow 0$$
yields $h^1(S,R_2\otimes\fini_{\eta})= \mathrm{length}(\eta)-h^0(S,R_2)+h^0(S,R_2\otimes\fini_{\eta})+h^1(S,R_2)$. Since $h^1(S,R_2)=0$ by \cite[Theorem]{knutsen2007sharp} and
$h^0(S,R_2\otimes\fini_{\eta})= h^0(S,R_2)-1$ since $R_2$ is globally generated, this gives us $h^1(S,R_2\otimes\fini_{\eta})=\mathrm{length}(\eta)-1=0$, a contradiction.

\paragraph{Case: $\length(\nu)=1$.} This case is similar to the previous case. Here we also have $R_1.R_2=0$, and in addition, $\length(\eta)=0$. Here, we cannot have $R_1^2>0$ with $R_2^2=0$,
because we get $R_2=\fino_S$, and dualising \eqref{M-sekvens} would imply that $\tilde{M}^{\vee}=M$ has global sections, a contradiction. So the two alternatives are $R_1^2<0$ or $R_1^2=0$,
as in the previous case. We cannot have $R_1^2<0$ for the same reason as in the previous case. If $R_1^2=0$ with $R_2^2>0$, we get $R_1=\fino_S$ as in the previous case, and $h^1(S,R_2)=1$.
However, since $R_2$ is globally generated with positive self-intersection, this is impossible by \cite[Theorem]{knutsen2007sharp}.

\paragraph{Case: $\length(\nu)=\length(\eta)=0$.} In this case, $R_1.R_2=1$, and since self intersection on a K3 surface is always even, Proposition \ref{HIT} yields that $R_1^2R_2^2\leq 0$. If
$R_1^2<0$, then it must be $\leq -2$, and we get $R_1.(R_1\otimes R_2)\leq -1$, and so $R_1$ is a base component of $c_1(\tilde{M})$, which contradicts $c_1(\tilde{M})$ being globally
generated. It
follows that $R_1^2\geq 0$. 

Having $R_1^2,R_2^2\geq 0$ implies that $(R_2\otimes R_1^{\vee})^2\geq -2$, and it follows from Riemann--Roch that either $R_2\geq R_1$ or $R_1\geq R_2$. By semistability of $\tilde{M}$, we have $R_1.L\leq \frac{1}{2}c_1(\tilde{M}).L=\frac{1}{2}(R_1\otimes R_2).L$. If $R_1\gneqq R_2$, this would give us $\frac{1}{2}(R_1\otimes R_2).L\lneqq R_1.L$ (by ampleness of $L$), which is impossible. It follows that $R_2\geq R_1$.

Now, since $R_2$ is globally generated, $(R_2\otimes R_1^{\vee}).R_2\geq 0$, and so $R_2^2\geq 2$ and $c_1(\tilde{M})^2\geq 4$. Since we originally had
$c_1(\tilde{M})^2\leq 4$ (as a consequence of assuming $c_2(\tilde{M})=1$), equality follows, together with $R_2^2=2$ and $R_1^2=0$. By \cite[Theorem]{knutsen2007sharp}, $h^1(S,R_2)=0$, and
since $h^2(S,\tilde{M})=0$, we get $h^2(S,R_1)=0$, and so $R_1=E^{\otimes n}$ where $E$ is an elliptic pencil. Since $R_1.R_2=1$, then $n=1$.

Note that since $R_2\otimes R_1^{\vee}>0$ and $\tilde{M}$ is semistable, then $\tilde{M}$ must be a non-split extension of $R_1$ and $R_2$. As we saw above, the dimension of isomorphism
classes of non-trivial extensions is $\mathrm{Ext}^1_{\fino_S}(R_2,R_1)-1=h^1(S,R_2\otimes R_1^{\vee})-1$, and so $h^1(S,R_2\otimes R_1^{\vee})>0$. By \cite[Theorem]{knutsen2007sharp}, this
implies that either $R_2\cong E\otimes E'$ where $E'$ is an elliptic pencil satisfying $E'.E=1$; or $(R_2\otimes R_1^{\vee}).\Gamma\leq -2$ for some $(-2)$-curve $\Gamma$, implying that
$R_2= E\otimes B\otimes \Gamma^{\otimes m}$, where $m\geq 1$ is an integer and $B\geq 0$ is a possibly trivial line bundle satisfying $B.\Gamma\geq 0$. 

If $R_2=E\otimes E'$, then $h^1(S,R_2\otimes R_1^{\vee})=0$, a contradiction.

If $R_2= E\otimes B\otimes \Gamma^{\otimes m}$, note that since $R_1=E$, then $1=R_1.R_2=E.B+mE.\Gamma\geq mE.\Gamma$. Since $(R_2\otimes R_1).\Gamma\geq 0$ (recall that $R_1\otimes R_2$ is
globally generated) and $(R_2\otimes R_1^{\vee}).\Gamma\leq -2$, we get $R_1.\Gamma\geq 1$. However, since $R_2.\Gamma\geq 0$ (recall that $R_2$ is globally generated), this means that
$(R_2\otimes R_1).\Gamma\geq 1$ instead of $\geq 0$, and we end up with $R_1.\Gamma\geq 2$. But then $R_2.R_1\geq mE.\Gamma\geq 2$, a contradiction.\\
\\
We conclude that $\Cliff(D_{|C})\leq\Cliff(A)$.
\end{proof}

\begin{proof}[Proof of Theorem \ref{main2}]
The proof of Theorem \ref{main2} uses exactly the same techniques as in the proof of Theorem \ref{main}. We include it here for the sake of completion.

The condition on $C$ and $A$ are that $\rho(g,1,d)<0$. In this case, it follows that $\finfca$ is non-simple, and hence non-stable.

Part (a) is proved using the same diagram as in \eqref{diagram}, the only difference being that $\rk(M)=1$, and that $M.L\geq -\frac{1}{2}L^2$ and $c_1(N).L\leq-\frac{1}{2}L^2$. The latter
inequality implies that $N$ has no global sections, and so $\phi$ is injective on global sections. It follows, from the arguments in the proof of Theorem \ref{main}, that each global
section of $A$ comes from a unique global section of $\fing$, and that the map must be the restriction map to $C$.

We now prove part (b): Following the proof of Theorem \ref{main} (b), it is clear that $h^0(S,D)\geq 2$. We have $h^0(S,L\otimes D^{\vee})=h^0(S,c_1(N)^{\vee})$. To prove that the latter is
$\geq 2$, is suffices by Riemann--Roch to show that $c_1(N)^2\geq 0$. We have $c_1(N).(M\otimes c_1(N))=c_1(N).L^{\vee}\geq\frac{1}{2}L^2=\frac{1}{2}M^2+M.c_1(N)+\frac{1}{2}c_1(N)^2$, and
so $\frac{1}{2}c_1(N)^2\geq \frac{1}{2}M^2$. Since $c_1(\tilde{M})=\Mvee$ and $\tilde{M}$ is globally generated away from a finite set, then $\Mvee$ is globally generated, and
so $M^2=(\Mvee)^2\geq 0$, and we can conlude that $h^0(S,L\otimes D^{\vee})\geq 2$. 

The argument that $h^0(C,D_{|C})$ is independent of the curve $C$ in $|L|_s$ is similar to the argument in the proof of Theorem \ref{main}. We see that no $(-2)$-curve can intersect $D$
negatively, and so $h^1(S,D)$ can be positive only if $D^2=0$. We see that $0=M.c_1(\fing)+c_2(\fing)$, and so $D^2=M^2\geq c_2(\fing)$. Thus, $D^2=0$ if and only if $\fing=D$. As a
consequence, $h^0(S,D)=2$, and so we must have $D=E$ where $E$ is an elliptic pencil. In that case, $h^1(S,D)=0$, and we conclude that $h^0(C,D_{|C})$ is independent of the curve $C$ in
$|L|_s$.

To prove (c), we have $\Cliff(A)=d-2=c_2(\finfca)-2$; and $\Cliff(D_{|C})=D.L-2(h^0(C,D_{|C})-1)\leq c_2(\finfca)-c_2(\tilde{M})-2\leq c_2(\finfca)-2$, and so $\Cliff(D_{|C})\leq\Cliff(A)$,
as desired.
\end{proof}

\section{An example of a linear system $|L|$ where $h^0(C,E^{\otimes 2}_{|C})$ depends on the curve $C$ in $|L|_s$}\label{example}

In this section, we give an example of a case where $h^0(C,E^{\otimes 2}_{|C})$ depends on the curve $C$ in $|L|_s$, thus showing that the condition in Theorem \ref{main} is necessary. We
first state the main result of this section before presenting a proposition and lemma to help us with the construction.

\begin{Theorem}\label{main3}
There exists a smooth K3 surface with an ample linear system $|L|$ and elliptic pencil $E$ such that $h^0(C,E^{\otimes 2}_{|C})$ depends on the curve $C$ in $|L|_s$.
\end{Theorem}

In order to obtain this, we will need a smooth K3 surface containing a $(-2)$-curve $\Gamma$ in addition to an elliptic pencil $E$ such that $E.\Gamma=2$. Furthermore, to ensure ampleness
of $L$, which we will define later, we need all divisors of $|E|$ to be irreducible. We start by proving that such a K3 surface exists.

\begin{Prop}\label{elliptic}
There exists a smooth K3 surface $S$ in $\tykkp^5$ which contains a $(-2)$-curve $\Gamma$ and an elliptic pencil $E$ such that $E.\Gamma=2$ and all divisors of $|E|$ are irreducible.
\end{Prop}

\begin{proof}
We will obtain the K3 surface $S$ by intersecting three quadric hypersurfaces $Q_1$, $Q_2$ and $Q_3$ in $\tykkp^5$ such that the intersection is smooth and irreducible, and where the hypersurfaces satisfy the following conditions:
\begin{itemize}
\item the quadric $Q_1$ contains a linear threespace $V_1$ and hence also a pencil of threespaces $V_2$ cut out by all hyperplanes $H$ containing $V_1$;
\item the threespace $V_1$ intersects $Q_2$ in the sum of two planes; and
\item the intersections of $Q_3$ with one of the planes in $V_1\cap Q_2$, and with $V_2\cap Q_2$ for all $V_2$, are smooth and irreducible.
\end{itemize}
We begin the proof by arguing that the above conditions will indeed produce line bundles with the desired properties. First of all, the intersection of $Q_3$ with one of the planes in $V_1\cap Q_2$ will be a plane conic, hence rational, and will be our desired $(-2)$-curve $\Gamma$. The pencil of threespaces $V_2$ will give us the desired elliptic pencil $E$, where clearly all divisors of $|E|$ will be irreducible. The curves in $|E|$ are elliptic because they are complete intersections of two quadrics in $\tykkp^3$, and the adjunction formula then gives us that the canonical bundle is trivial. We compute the intersection $E.\Gamma$ using the equality $(E+\Gamma)^2=(H-\Gamma_2)^2$, where $\Gamma_2$ is the other plane conic in $V_1\cap Q_2\cap Q_3$. We have $(E+\Gamma)^2=2E.\Gamma-2$, and $(H-\Gamma_2)^2=2$, using that $H.\Gamma_2=2$. It follows that $E.\Gamma=2$.

We now construct $Q_1$, $Q_2$ and $Q_3$ with the above-listed properties. The idea of the construction is to build cones over quadric surfaces $Z_i$ contained in linear threespaces $J_i$, with certain lines $T_i$ as base loci.

We start with $Q_1$. Let $J_1$ be any linear threespace, and let $T_1$ be a line that avoids $J_1$. Consider the linear system of quadric surfaces $Z_1$ in $J_1$. For given $Z_1$, we let $Q_1$ be the span of planes connecting $T_1$ with points on $Z_1$. This quadric fourfold is smooth outside of $T_1$ because of the following: By blowing up in $T_1$ and considering the strict transform $\tilde{Q}_1$ of $Q_1$, we see that by varying $Z_1$, the linear system we obtain in $\tykkp^5\times\tykkp^3$ is base-point free, since the preimage of $T_1$ corresponds to planes connecting $T_1$ with points in $J_1$. Thus, the general $\tilde{Q}_1$ is smooth by Bertini's theorem, and so the general $Q_1$ is smooth outside of $T_1$. We choose the desired threespace $V_1$ to be the span of $T_1$ with any line on $Z_1$. The pencil of threespaces $V_2$ is, as previously mentioned, cut out by considering hyperplanes $H$ containing $V_1$.

Before constructing $Q_2$, we first choose a line $T_3$ not contained in $Q_1$ and which avoids $V_1$, and a linear threespace $J_3$ avoiding $T_3$ and chosen such that $Q_1\cap J_3$ is smooth and irreducible (which can be done by considering general $\tilde{J}_3$ on $\tilde{Q}_1$). These will be used to construct $Q_3$ later. Note that $T_3$ intersects $Q_1$ in two points.

We now construct $Q_2$: First, let $T_2$ be a line intersecting $V_1$ in a point; avoids $T_1$, $T_3$ and the two $V_2$'s that intersect $T_3$; and is not contained in $Q_1$. We next consider all possible threespaces $J_2$ and note that all intersect $V_1$ in at least a line $\ell$. For each $J_2$ and each line $\ell$ in $J_2\cap V_1$, consider the linear system of quadric surfaces $Z_2$ containing $\ell$. The general $J_2$ will intersect $V_1$ in only a single line $\ell$, and will avoid $T_2$ and $T_1$. For given $J_2$ and $Z_2$, we let $Q_2$ be given by the span of planes connecting $T_2$ with points on $Z_2$. The threespace $V_1$ then intersects $Q_2$ in two planes, where one is given by the span of $V_1\cap T_2$ and the line $\ell$. The general $Q_2$ intersects $J_3$ in a smooth irreducible surface, since the various $Q_2$ define a base-point free linear system on $J_3$ and we can thus apply Bertini's theorem.

It now remains to construct $Q_3$. We have already chosen $T_3$ and $J_3$. Consider the linear system of quadric surfaces $Z_3$ in $J_3$, and note that the general $Z_3$ intersects $Q_2$ in a smooth irreducible curve. We now construct $Q_3$ just as with $Q_1$ and $Q_2$ and connect planes between $T_3$ and points in $Z_3$. The resulting linear system of $Q_3$'s cut out a base-point free linear system on the two planes in $V_1\cap Q_2$. (Proof: Let $\Pi$ be one of the planes. Since $T_3$ avoids $V_1$, it in particular avoids $\Pi$. Each point on $\Pi$ and $T_3$ span a plane, which intersects $J_3$ in a point. We can then choose $Z_3$ so that it avoids that point.) By Bertini's theorem, the general $Q_3$ cuts out a smooth conic on each of the planes in $V_1\cap Q_2$.

To show that every element $V_2\cap Q_2\cap Q_3$ of the elliptic pencil $|E|$ is irreducible, first consider the $V_2$'s that avoid $T_3$. It is then clear that $V_2\cap Q_2\cap Q_3$ is found by first intersecting $Z_3$ with $Q_2$, yielding a curve $C$ in $Z_3$, and then considering the projection map to $V_2$ by planes connecting points on $C$ with $T_3$. This projection map is clearly an isomorphism, and so as long as $C$ is irreducible, then the corresponding curve in $V_2$ is also irreducible. In the two cases when $V_2$ intersects $T_3$, these threespaces avoid $T_2$, by how $T_2$ was chosen, and so we can consider the curve $C'=Z_2\cap Q_3$ and construct a similar projection map using planes through $T_2$. Since for general $Q_2$ and $Q_3$ we have $Z_2\cap Q_3$ irreducible and $Z_3\cap Q_2$ irreducible, we are done.

It remains to show that $Q_1\cap Q_2\cap Q_3$ is smooth. First of all, by blowing up in $T_1$, we have already seen that the general $\tilde{Q}_1$ is smooth. We now also blow up in $T_2$ and consider the linear system of $\tilde{\tilde{Q}}_2$'s on $\tilde{\tilde{Q}}_1$. Using the same argument as in the $T_1$ case (which is possible since the general $J_2$ avoids $T_1$), we see that the linear system of $\tilde{\tilde{Q}}_2$'s is also base-point free, and so the general element defines a smooth divisor on $\tilde{\tilde{Q}}_1$. We can choose $Q_2$ such that $\tilde{\tilde{Q}}_1\cap\tilde{\tilde{Q}}_2$ avoids $T_3$, and so the linear system of $\tilde{\tilde{Q}}_3$'s restricted to $\tilde{\tilde{Q}}_1\cap\tilde{\tilde{Q}}_2$ will then also be base-point free, and the general $\tilde{\tilde{Q}}_1\cap\tilde{\tilde{Q}}_2\cap\tilde{\tilde{Q}}_3$ is therefore smooth. The general $Q_1\cap Q_2\cap Q_3$ avoids $T_1$, $T_2$ and $T_3$, and is thus isomorphic to $\tilde{\tilde{Q}}_1\cap\tilde{\tilde{Q}}_2\cap\tilde{\tilde{Q}}_3$. The result follows.
\end{proof}

In the following, we will let $S$ be a K3 surface as given in Proposition \ref{elliptic}. We now consider the line-bundle $L=E^{\otimes a}\otimes\Gamma^{\otimes
(a-1)}$ where $a\geq 3$ is an integer. (We let $a\geq 7$ if we wish the condition $\deg(A)\leq\frac{1}{6}L^2$ from Theorem \ref{main} to be satisfied, but our example works for all $a\geq
3$.) In the following lemma, we prove that $|L|$ contains smooth irreducible curves, and that $L$ is ample.

\begin{Lemma}\label{ample}
Let $L$ and $S$ be as above. Then:
\begin{itemize}
\item[(a)] the linear system $|L|$ contains smooth, irreducible curves; and
\item[(b)] the line-bundle $L$ is ample.
\end{itemize}
\end{Lemma}

\begin{proof}
In part (a), we prove that $\Gamma$ is not a base component of $|L|$. Then, since linear systems on K3 surfaces without base components are base-point free, we can apply Bertini's theorem
to conclude smoothness for the general curves. In order to prove that $\Gamma$ is not a base component, we show that $h^0(S,L)>h^0(S,L\otimes \Gamma^{\vee})$. Using
\cite[Theorem]{knutsen2007sharp}, $h^1(S,L)=h^1(S,L\otimes \Gamma^{\vee})=0$, and so $h^0(S,L)=\frac{1}{2}L^2+2$ and $h^0(S,L\otimes
\Gamma^{\vee})=\frac{1}{2}(L\otimes\Gamma^{\vee})^2+2=h^0(S,L)-L.\Gamma-1=h^0(S,L)-3$. It follows that the general element of $|L|$ is smooth and irreducible.

To prove (b), we first of all have $L.E>0$ and $L.\Gamma>0$. It is also clear that $L.D\geq 0$ for any other irreducible curve
    $D$, since $|L|$ contains irreducible curves. Suppose $L.D=0$ for some $D$. Since $D$ can have neither $E$ nor $\Gamma$ as a component, it follows that $E.D=\Gamma.D=0$.
    We now use Proposition \ref{HIT} with $L=D_2$ and $D=D_1$, and see that $D^2\leq 0$. Suppose first that
    $D^2=0$. Using Proposition \ref{HIT} again, this time with $D_2=L$ and $D_1=L\otimes D^{\vee}$, it follows that $D$ must be trivial.

Now suppose $D^2<0$. Since $D$ is irreducible, this means that $D$ is a $(-2)$ curve. Consider the exact sequence
$$0\rightarrow E^{\vee}\rightarrow \fino_S\rightarrow \fino_{E'}\rightarrow 0,$$
where $E'$ is a smooth elliptic curve in $|E|$. Tensor the sequence with $D$ and take cohomology. We have $h^1(S,D)=0$ and $h^1(E',D_{|E'})=1$, since $D_{|E'}$ is trivial (recall that
$D.E=0$) and $\omega_{E'}\cong\fino_{E'}$. This gives us
    $h^2(S,D\otimes E^{\vee})=h^0(S,E\otimes D^{\vee})=1$, implying that $D<E$, which contradicts Proposition \ref{elliptic}. We can thus conclude that $L$ is ample. 
\end{proof}

We now prove the main result of this section. The idea is to consider divisors in $|E^{\otimes 2}\otimes\Gamma|$ restricted to various curves $C$ in $|L|_s$ and subtract $\Gamma\cap C$. For
the general curves, we get the same linear system as from $|E^{\otimes 2}|$. However, we show that there exist other curves $C'$ where divisors in $|E^{\otimes 2}\otimes \Gamma|$ without
$\Gamma$ as a component intersect $C'$ exactly in the points $\Gamma\cap C'$, and the linear system on $C'$ thus obtains an extra dimension.

\begin{proof}[Proof of Theorem \ref{main3}]
Consider the exact sequence
\begin{equation}\label{2E-sekvens}
0\rightarrow L^{\vee}\otimes E^{\otimes 2}\rightarrow E^{\otimes 2}\rightarrow E^{\otimes 2}_{|C}\rightarrow 0,
\end{equation}
where $C$ is a curve in $|L|_s$. We argue that $h^1(S,L^{\vee}\otimes E^{\otimes 2})=1$. Note that this equals $h^1(S,L\otimes E^{\otimes (-2)})$, and by \cite[Theorem]{knutsen2007sharp},
$h^1(S,L\otimes E^{\otimes (-2)})>0$ while $h^1(S,L\otimes E^{\otimes (-2)}\otimes \Gamma^{\vee})=0$. By comparing $\chi(S,L\otimes E^{\otimes (-2)})$ with $\chi(S,L\otimes E^{\otimes
(-2)}\otimes\Gamma^{\vee})$, and using that $h^0(S,L\otimes E^{\otimes (-2)})=h^0(S,L\otimes E^{\otimes (-2)}\otimes\Gamma^{\vee})$, we get that $h^1(S,L\otimes E^{\otimes (-2)})=1$.

Now tensor \eqref{2E-sekvens} with $\Gamma$ and take cohomology. We see that $H^0(S,E^{\otimes 2}\otimes\Gamma)\cong H^0(C,E^{\otimes 2}\otimes \Gamma_{|C})$, and we can therefore conclude
that the linear system $|E^{\otimes 2}_{|C}|$ is found precisely by considering divisors in $|E^{\otimes 2}\otimes \Gamma|$ that, restricted to $C$, are zero in $\Gamma\cap C$.

Note that divisors in $|E^{\otimes 2}\otimes \Gamma|$ that have $\Gamma$ as a component will not cut out any extra divisors in $|E^{\otimes 2}_{|C}|$ apart from those already cut out by
$|E^{\otimes 2}|$ on $S$. We must therefore consider curves in $|E^{\otimes 2}\otimes \Gamma|$ that do not have $\Gamma$ as a component, but still cut through $C$ exactly where $C$
intersects $\Gamma$.

We now have two situations: First of all, consider curves $J$ in $|E^{\otimes 2}\otimes \Gamma|$ that do not have $\Gamma$ as a component. By considering the exact sequence
\begin{equation}\label{2e+gamma}
0\rightarrow E^{\otimes 2}\rightarrow E^{\otimes 2}\otimes \Gamma\xrightarrow{\psi} (E^{\otimes 2}\otimes \Gamma)_{|\Gamma}\rightarrow 0
\end{equation}
and noting that $h^1(S,E^{\otimes 2})=1$ while $h^1(S,E^{\otimes 2}\otimes \Gamma)=0$, we see that by varying $J$, one dimension of divisors in $|(E^{\otimes 2}\otimes
\Gamma)_{|\Gamma}|=|\fino_{\Gamma}(2)|$ is cut out. The image of $\psi$ cuts out a sub-linear system of $|\fino_{\Gamma}(2)|$ that we denote by $\mathfrak{d}$. Since $|E^{\otimes
2}\otimes\Gamma|$ is base-component free and hence base-point free, it follows that $\mathfrak{d}$ is base-point free. Now consider the exact sequence
$$0\rightarrow L\otimes\Gamma^{\vee}\rightarrow L\xrightarrow{\tau} L_{|\Gamma}\rightarrow 0,$$
and note that $h^1(S,L\otimes\Gamma^{\vee})=0$, by \cite[Theorem]{knutsen2007sharp}. We see here that $|L|$ cuts out the entire linear system $|\fino_{\Gamma}(2)|$, and we can then consider
the pre-image of the sub-linear system $\mathfrak{d}$, which then is a sub-linear system $\mathfrak{L}$ of $|L|$ without base-points along $\Gamma$. We prove that $\mathfrak{L}$ is also
base-point free everywhere else, so that we can apply Bertini's theorem and conclude that the general element of $\mathfrak{L}$ is smooth.

The proof is by induction on $a$ in the expression $L=E^{\otimes a}\otimes \Gamma^{\otimes (a-1)}$, where we start with $a=2$. In this case, the curves in $L$ and in $|E^{\otimes 2}\otimes
\Gamma|$ are the same, so the result is obvious. Now for each $a\geq 3$, suppose there is a base-point free sub-linear system $\mathfrak{L}'$ of $|E^{\otimes (a-1)}\otimes\Gamma^{\otimes
(a-2)}|$ consisting of divisors that intersect $\Gamma$ where the curves $J$ intersect. Consider the linear system $|E\otimes\Gamma|$. By Riemann--Roch, it contains divisors that do not
have $\Gamma$ as a component, and so it is clearly base-component free and hence base-point free. Also, since $(E\otimes\Gamma).\Gamma=0$, the irreducible elements of this linear system
avoids $\Gamma$. Denote these elements by $V$. It then follows that $\mathfrak{L}'+V$ is a subset of divisors in $|E^{\otimes a}\otimes\Gamma^{\otimes (a-1)}|$ that intersect $\Gamma$ where
$J$ intersects. It is clearly base-point free, and so we can conclude that the sub-linear system $\mathfrak{L}$ is also base-point free. By Bertini's
theorem, it follows that the general element of $\mathfrak{L}$ is smooth.

We conclude that for general $J\in |E^{\otimes 2}\otimes\Gamma|$ without $\Gamma$ as a component, and smooth curves $C'$ in $|L|$ that pass through $J\cap\Gamma$, there is an effective
divisor $J\cap C'-\Gamma\cap C'\in|E^{\otimes 2}_{|C'}|$ which is not cut out by a divisor in $|E^{\otimes 2}|$. As a result, we get $h^0(C',E^{\otimes 2}_{|C'})=4$.

It remains to prove that there exists a smooth, irreducible curve $C''\in |L|$ where the above situation does not occur. It then suffices to find a curve $C''$ such that no divisor $J$ in
$|E^{\otimes 2}\otimes \Gamma|$ satisfies $J\cap\Gamma=C''\cap \Gamma$. This follows by the exact sequence \eqref{2e+gamma}, where we saw that $\psi$ is not surjective on global sections
and hence that the general divisors $Z\in |(E^{\otimes 2}\otimes \Gamma)_{|\Gamma}|$ are not cut out by any of the divisors in $|E^{\otimes 2}\otimes \Gamma|$. From the argument above, it
follows that any curve $C''$ that cuts out such a divisor $Z$ on $\Gamma$ will satisfy $h^0(C'',E^{\otimes 2}_{|C''})=3$.
\end{proof}

\paragraph{Acknowledgements.} Thanks to Shengtian Zhou for helpful comments. This work was done during our stay at the University of Utah, and we are grateful for their warm hospitality. The
author would also like to thank the referee for helpful suggestions.

This work was funded by the research stipend of Telemark University College.

\bibliographystyle{plain}
\bibliography{NHR}

\end{document}